\documentclass[reqno,11pt]{amsart}
\usepackage{amsmath,amssymb,latexsym,soul,cite,mathrsfs,amsfonts}
\usepackage{color,enumitem,graphicx}
\usepackage[colorlinks=true,urlcolor=blue,
citecolor=red,linkcolor=blue,linktocpage,pdfpagelabels,
bookmarksnumbered,bookmarksopen]{hyperref}
\usepackage[english]{babel}
\usepackage{mwe}
\usepackage{comment}
\usepackage{tikz}
\usepackage[left=2.9cm,right=2.9cm,top=2.8cm,bottom=2.8cm]{geometry}
\usepackage[hyperpageref]{backref}

\usepackage[colorinlistoftodos]{todonotes}
\makeatletter
\providecommand\@dotsep{5}
\def\listtodoname{List of Todos}
\def\listoftodos{\@starttoc{tdo}\listtodoname}
\makeatother

\usepackage{verbatim} 
\usepackage{tikz}

\numberwithin{equation}{section}

\def\cal{\mathcal}
\newtheorem{theorem}{Theorem}[section]
\newtheorem{proposition}[theorem]{Proposition}
\newtheorem{lemma}[theorem]{Lemma}
\newtheorem{corollary}[theorem]{Corollary}
\newtheorem{remark}[theorem]{Remark}

\newtheorem{example}{Example}

\newcommand\restr[2]{{
  \left.\kern-\nulldelimiterspace 
  #1 
  \vphantom{\big|} 
  \right|_{#2} 
  }}

\title[Further applications of the Nehari manifold method to functionals in $C^1(X \setminus \{0\})$]
{Further applications of the Nehari manifold method to functionals in $C^1(X \setminus \{0\})$}

\author[E.J.F. Leite]{Edir Júnior Ferreira Leite}
\author[H. Ramos Quoirin]{Humberto Ramos Quoirin}

\author[K. Silva]{Kaye Silva}

\address[E. J.F. Leite]{\newline\indent
	Departamento de Matem\'atica.   
	\newline\indent 
	Universidade Federal de S\~ao Carlos,
	\newline\indent
13565-905, S\~ao Carlos, SP, Brazil}
\email{\href{mailto:edirleite@ufscar.br}{edirleite@ufscar.br}}

\address[H. Ramos Quoirin]{\newline \indent CIEM-FaMAF \newline \indent Universidad Nacional de C\'{o}rdoba, \newline\indent 
(5000) C\'{o}rdoba, Argentina}
\email{\href{mailto:humbertorq@gmail.com}{ humbertorq@gmail.com}}

\address[K. Silva]{\newline\indent
	Instituto de Matem\'atica e Estat\'istica.   
	\newline\indent 
	Universidade Federal de Goi\'as,
	\newline\indent
Rua Samambaia, 74001-970, Goi\^ania, GO, Brazil}
\email{\href{mailto:kayesilva@ufg.br}{kayesilva@ufg.br}}

\thanks{
E. J. F. Leite was partially supported by CNPq/Brazil under grant 316526/2021-5. H. Ramos Quoirin was
partially supported by FAPEG Programa Pesquisador Visitante Estrangeiro 2024. Kaye Silva was partially supported by CNPq/Brazil under Grants 308501/2021-7 and 201334/2024-0. This work was completed while the third author held a post-doctoral position at Florida Institute of Technology, Melbourne, United States of America}


\begin{document}

\begin{abstract}
We proceed with the study of the Nehari manifold method for functionals in $C^1(X \setminus \{0\})$, where $X$ is a Banach space. We deal now with functionals whose fibering maps have two critical points (a minimiser followed by a maximiser).  Under some additional conditions we show that the Nehari manifold method provides us with the ground state level and two sequences of critical values for these functionals. These results are applied to the class of {\it prescribed energy problems} as well as to the concave-convex problem for the {\it affine} $p$-Laplacian operator.
 \end{abstract}

\bigskip

\maketitle

\bigskip

 \section{Introduction}
 \medskip
Let $X$ be an infinite-dimensional uniformly convex (real) Banach space, equipped with $\|\cdot \| \in C^1(X \setminus \{0\})$. In \cite{LRQS} we have considered the Nehari set associated to $\Phi \in C^1(X \setminus \{0\})$, namely
$$\mathcal{N}=\mathcal{N}(\Phi):=\{u \in X \setminus \{0\}: \Phi'(u)u=0\},$$
assuming that for every $u \in X \setminus \{0\}$ the {\it fibering map} $t \mapsto \Phi(tu)$, $t>0$, has a unique critical point $t(u)>0$. Following the approach used in \cite{SW}, we noted that under some conditions on the map $u \mapsto t(u)$ the set $\mathcal{N}$ is homeomorphic to $\mathcal{S}$, the unit sphere in $X$, and there is a one-to-one correspondence between critical points of $\Phi$ and critical points of the map $u \mapsto \Phi(t(u)u)$, $u \in \mathcal{S}$. Assuming $\Phi$ to be even and satisfy some standard compactness conditions one may then apply the Ljusternik-Schnirelman theory  to obtain a sequence of critical values of $\Phi$, the first term of this sequence being its ground state level,  cf. \cite[Theorem 2.1]{LRQS}.

We shall consider now the case where the fibering map associated to $\Phi$ has two critical points. More precisely we assume the following condition:
 \begin{itemize}
 	\item[(H1)] For every $u \in X \setminus \{0\}$ the map $t \mapsto \Phi(tu)$, defined for $t>0$, has exactly two critical points $t^+(u)<t^-(u)$, with $t^+(u)$ being a local minimum point and $t^-(u)$ a local maximum point . Moreover:
	\begin{enumerate}
	\item the map $u \mapsto t^+(u)$ is bounded (away from zero and from above) in any compact subset of $\mathcal{S}$.
	\item the map $u \mapsto t^-(u)$ is bounded away from zero in $\mathcal{S}$, and bounded from above in any compact subset of $\mathcal{S}$. \\	
\end{enumerate}
	 
 \end{itemize}

Figure 1 below shows some examples of fibering maps satisfying this condition.
From (H1) it follows that $$\mathcal{N}=\mathcal{N}^+ \cup \mathcal{N}^-, \quad \mbox{where} \quad \mathcal{N}^{\pm}=\mathcal{N}^{\pm}(\Phi):=\{t^{\pm}(u)u: u \in \mathcal{S} \}.$$
Note that $\mathcal{N}^+$ and $\mathcal{N}^-$ are not necessarily manifolds since these sets may intersect the degenerate part of $\mathcal{N}$, namely the set
$$\mathcal{N}^0=\mathcal{N}^0(\Phi):=\left\{tu \in X \setminus \{0\}: \frac{d}{dt} \Phi(tu)=\frac{d^2}{dt^2} \Phi(tu)=0\right\}$$
in case the map $t \mapsto \Phi(tu)$ is twice differentiable for some $u \in X \setminus \{0\}$.
Thus our definition of $\mathcal{N}^+$ and $\mathcal{N}^-$ is different from the one usually adopted in the litterature, which excludes points in $\mathcal{N}^0$. Such splitting of the Nehari set
has been used for several functionals associated to nonlinear boundary value problems (its first occurences being \cite{Ta} to the best of our knowledge).  Besides \cite{BWu}, which proves the existence of minimisers of $\Phi$ constrained to $\mathcal{N}^+ \setminus \mathcal{N}^0$ and $\mathcal{N}^-\setminus \mathcal{N}^0$ for a particular class of functionals, we are not aware of a general abstract setting of this splitting, even for $\Phi \in C^1(X)$. 

 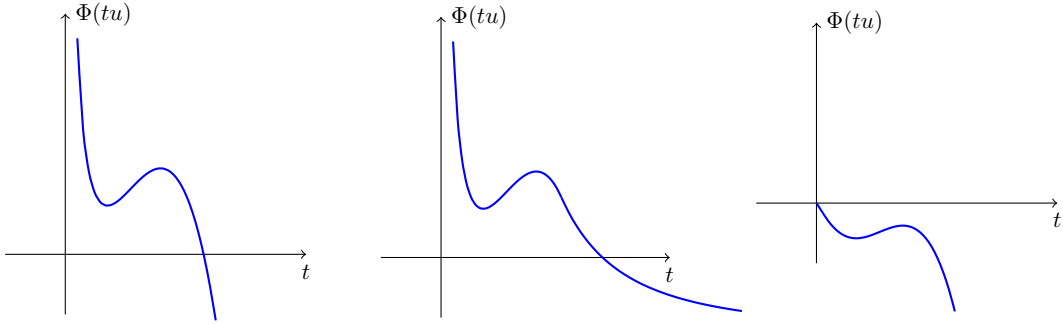
\begin{figure}[h]
	\centering
	\begin{minipage}{0.2\textwidth} 
		\begin{tikzpicture}[scale=0.8]
		\draw[->] (-1,0) -- (4,0) node[below] {\scalebox{0.8}{$t$}};
		
		\draw[->] (0,-1) -- (0,4) node[right] {\scalebox{0.8}{$\Phi(tu)$}};
		
		\draw[blue,thick,domain=0.2:2.5,smooth,variable=\x] plot ({\x},{(\x)^(-1)+2.5*(\x)^(2)-(\x)^(3)-1.5});	
		
		\end{tikzpicture}

	\end{minipage}
	\hspace{.1\linewidth}
	\begin{minipage}{0.2\textwidth} 
		\begin{tikzpicture}[scale=0.8]
		\draw[->] (-1,0) -- (3.8,0) node[below] {\scalebox{0.8}{$t$}};
		
		\draw[->] (0,-1) -- (0,4) node[right] {\scalebox{0.8}{$\Phi(tu)$}};
		
		\draw[blue,thick,domain=0.2:2,smooth,variable=\x] plot ({\x},{(\x)^(-1)+2.5*(\x)^(2)-(\x)^(3)-1.5});	
		
		\draw[blue,thick,domain=2:5,smooth,variable=\x] plot ({\x},{9*(\x)^(-2)-1.25});	
		
		\end{tikzpicture}

	\end{minipage}
	\hspace{.1\linewidth}
	\begin{minipage}{0.2\linewidth} 
		\begin{tikzpicture}[scale=0.8]
		\draw[->] (-1,0) -- (4,0) node[below] {\scalebox{0.8}{$t$}};
		
		\draw[->] (0,-1) -- (0,3) node[right] {\scalebox{0.8}{$\Phi(tu)$}};
		
		\draw[blue,thick,domain=0:2.3,smooth,variable=\x] plot ({\x},{-3*(\x)^(1.2)+3.5*(\x)^(2)-(\x)^(3)});	
		
		\end{tikzpicture} 
	\end{minipage}
	\caption{Some possible behaviors of the map $t \mapsto \Phi(tu)$ under (H1).} \label{fig2}
\end{figure}

We shall see that (H1) implies that $\mathcal{N}^+$  and $\mathcal{N}^-$ are both homeomorphic to $S$, so that the Ljusternik-Schnirelman theory applies to $\Phi|_{\mathcal{N}^\pm}$ under the following condition:\\

 \begin{itemize}
	\item[$(H2)^{\pm}$]  $\Phi$ is even, bounded from below on $\mathcal{N}^{\pm}$, and satisfies the Palais-Smale condition on $\mathcal{N}^{\pm}$, i.e.
 	any sequence $(u_n) \subset \mathcal{N}^{\pm}$ such that $(\Phi(u_n))$ is bounded and $\Phi'(u_n) \to 0$ has a convergent subsequence in $\mathcal{N}^{\pm}$.\\
 \end{itemize}

Indeed, it provides us with two sequences of critical values of $\Phi$, namely $(\lambda_n^+)$ and $(\lambda_n^-)$, which are given by $$\lambda_n^{\pm}:=\displaystyle
\inf_{F\in \mathcal{F}_n}\sup_{u\in F}\Phi(t^{\pm}(u)u),$$ where $
\mathcal{F}_n:=\{F\subset \mathcal{S}: F \mbox{ is compact, symmetric, and } \gamma(F)\ge n\}$, and $\gamma(F)$ is the Krasnoselskii genus of $F$. 
Our main abstract result reads as follows:

\begin{theorem}\label{tn}
Assume (H1) and $(H2)^{\pm}$. Then $(\lambda_n^{\pm})$ is a nondecreasing unbounded sequence of critical values of $\Phi$. Moreover $\lambda_1^+$ is the ground state level of $\Phi$, i.e. its least critical level, $\lambda_1^+<\lambda_1^-$, and $\lambda_n^+ \le \lambda_n^-$ for any $n \ge 2$.
\end{theorem}

\begin{remark} In $(H2)^{\pm}$ it is enough to assume that the Palais-Smale condition holds at the level $\lambda_n^{\pm}$ to deduce that $\lambda_n^{\pm}$ is a critical value of $\Phi$. On the other hand, the global Palais-Smale condition allows us to deduce that $\lambda_n^{\pm} \to \infty$.
\end{remark}







Next we consider two applications of Theorem \ref{tn}:
\subsection{The prescribed energy problem}
As already explained in \cite{LRQS}, our motivation for dealing with functionals in $C^1(X \setminus \{0\})$ comes from several problems, the first one being
 the {\it prescribed energy problem} for the family of functionals  $\Phi_\lambda:=I_1-\lambda I_2$, where $\lambda$ is a real parameter and $I_1,I_2\in C^1(X)$. We consider the problem
\begin{equation}\label{ef}
\Phi_\lambda'(u)=0, \quad \Phi_{\lambda}(u)=c,
\end{equation}
where $c$ is an arbitrary real number.
 By a (nontrivial) solution of \eqref{ef} we mean a couple $(\lambda,u) \in \mathbb{R} \times X \setminus \{0\}$ satisfying \eqref{ef}, in which case $u$ is a critical point of $\Phi_\lambda$ at the level $c$.  In this sense, a solution of \eqref{ef} yields a critical point of  $\Phi_\lambda$ with prescribed energy. 
As discussed in \cite{LRQS}(see also \cite{I3,RSiS}), one can solve \eqref{ef} by dealing with the functional $\lambda_c \in C^1(X \setminus \{0\})$ given by 
\begin{equation}\label{dlc}
\lambda_c(u):=\frac{I_1(u)-c}{I_2(u)}, \quad u \in X \setminus \{0\},
\end{equation}
where we assume that $I_2(u) \neq 0$ for every $u \in X \setminus \{0\}$. In addition we assume that $I_2$ is $\alpha$-homogeneous for some $\alpha>0$, i.e. $I_2(tu)=t^{\alpha} I_2(u)$ for any $t>0$ and $u \in X$.

Our main model for this situation is the functional \begin{equation}\label{fm}
\Phi_\lambda(u)=\frac{1}{2}\int_{\Omega} |\nabla u|^2-\frac{\lambda}{2}\int_{\Omega} |u|^q-\int_{\Omega} F(u), \quad u \in H_0^1(\Omega),
\end{equation} 
which is associated to the boundary value problem
\begin{equation}\label{semi}
\begin{cases}
-\Delta u=\lambda |u|^{q-2}u +f(u) &\mbox{ in } \Omega,\\
u=0 &\mbox{ on } \partial \Omega.
\end{cases}
\end{equation}
Here $\Omega \subset \mathbb{R}^N$ ($N \geq 1$) is a bounded domain, $1<q<2$, $F(s):=\int_0^s f(t) dt$, and $f\in C(\mathbb{R})$ has superlinear and subcritical Sobolev growth. In this case the functional $\lambda_c$ is given by
 \[
\lambda_c(u)=\frac{\frac{1}{2}\int_\Omega |\nabla u|^2-\int_\Omega F(u)-c}{\frac{1}{q}\int_\Omega |u|^q}, \quad u \in H_0^1(\Omega) \setminus \{0\}.
\]
We shall see that this functional satisfies (H1), (H2), and (H3) if $c$ is negative and close enough to $0$, which produces two unbounded sequences of critical values of $\lambda_c$.

Let $\lambda_c$ be given by \eqref{dlc}. We know that (cf. \cite[Subsection 2.1]{LRQS})
\[
\mathcal{N}(\lambda_c)=\{u \in X  \setminus \{0\}: H(u)+ \alpha c=0\},
\]
where
$$H(u):=I_1'(u)u-\alpha I_1(u), \quad u \in X.$$

The behavior of the map $t \mapsto H(tu)$ is then crucial to understand $\mathcal{N}(\lambda_c)$. The following condition on $H$ implies that the  map  $t \mapsto \lambda_c(tu)$ has two critical points if $c<0$ is close enough to zero:

\begin{itemize}
	\item[(F1)] For any $u \in X \setminus \{0\}$ there exists $s(u) > 0$ such that the map $t \mapsto H(tu)$ is increasing in $(0,s(u))$ and decreasing in $(s(u),\infty)$. Moreover $\displaystyle \lim_{t \to \infty} H(tu)=-\infty$ uniformly on weakly compact subsets of $X \setminus \{0\}$.\\ 
\end{itemize}

Indeed, since $H(0)=0$ it follows that $s(u)$ is the global maximum point of the map $t \mapsto H(tu)$ for any $u \in X  \setminus \{0\}$, so that the equation $H(tu)+ \alpha c=0$ has two solutions $t^+(u)<t^-(u)$ if $H(s(u)u)+\alpha c>0$. Note also that, under some additional conditions,  
$H$ has a ground state level 	\begin{equation}
\label{cal}
h_0:=\inf_{u \in X \setminus \{0\}}\max_{t>0} H(tu)=\inf_{u\in \mathcal{S}}H(s(u)u)=\inf_{\mathcal{N}(H)} H.
\end{equation} This is the case, for instance, if $$H(u)=\frac{2-q}{2}\int_\Omega |\nabla u|^2+\int_\Omega \left(qF(u)-f(u)u\right), \quad  u \in H_0^1(\Omega),$$
which is associated to \eqref{semi}, cf. \cite[Proof of Corollary 4.5]{LRQS}.

The next condition will be used to show that $\lambda_c$ satisfies the Palais-Smale condition on $\mathcal{N}(\lambda_c)$:

\begin{itemize}
	\item[(F2)] $I_1=J-K$, where $J,K \in C^1(X)$ are such that $K'$ is completely continuous, and there exist $C_1,C_2>0$, $\beta \geq\alpha$, and $\eta>1$ such that $J(u)\geq C_1\|u\|^{\beta}$ and $\left(J'(u)-J'(v)\right)(u-v)\geq C_2(\|u\|^{\eta-1}-\|v\|^{\eta-1})(\|u\|-\|v\|)$ for any $u,v \in X$. 
\end{itemize}

We are now in position to apply Theorem \ref{tn} to $\lambda_c$, which yields the following result:

\begin{theorem}\label{c1}
	Let  $I_1,I_2 \in C^1(X)$ be such that $I_1(0)=I_2(0)=0$, $I_2(u)>0$ for $u \in  X \setminus \{0\}$,  $I_2$ is $\alpha$-homogeneous for some $\alpha>1$, and $I_2'$ is completely continuous. Assume in addition that $I_1$ satisfies $(F1)$, $(F2)$,  and 
	$0<-\alpha  c<h_0$, where $h_0$ is given by \eqref{cal}.
	Then:
	\begin{enumerate}
		\item $\lambda_c$  has a ground state level, given by $\lambda_{1,c}^+=\displaystyle \min_{ \mathcal{N}^+(\lambda_c) } \lambda_c$. 
		\item If, in addition, $I_1,I_2$ are even then $\lambda_c$ has two nondecreasing unbounded sequences of critical values $(\lambda_{n,c}^+)$ and $(\lambda_{n,c}^-)$, given by
		$\displaystyle \lambda_{n,c}^{\pm}:=\displaystyle
		\inf_{F\in \mathcal{F}_n}\sup_{u\in F}\lambda_c(t_c^{\pm}(u)u)$.
		 Moreover $\lambda_{1,c}^+ < \lambda_{1,c}^-$ and $\lambda_{n,c}^+ \leq \lambda_{n,c}^-$ for every $n \ge 2$.
	\end{enumerate}
\end{theorem}

Let us restate Theorem \ref{c1} in terms of the functional $\Phi_\lambda$:
\begin{corollary}\label{cc1}
Under the assumptions of Theorem \ref{c1}, the following assertions hold:
\begin{enumerate}
	\item $\Phi_\lambda$ has a critical point $u_{1,c}$ at the level $c$ for $\lambda=\lambda_{1,c}^+$, and has no such critical point for $\lambda<\lambda_{1,c}^+$.
	\item If, in addition, $I_1,I_2$ are even, then for every $n\in \mathbb{N}$ the functional $\Phi_\lambda$ has, at the level $c$, a pair of critical points $\pm u_{n,c}$  for $\lambda=\lambda_{n,c}^+$, and a pair of critical points $\pm v_{n,c}$ for $\lambda=\lambda_{n,c}^-$.
	 Moreover $\lambda_{1,c}^+ < \lambda_{1,c}^-$, $\lambda_{n,c}^+ \leq \lambda_{n,c}^-$ for every $n \ge 2$, and $\lambda_{n,c}^{\pm} \nearrow \infty$ as $n \to \infty$.
\end{enumerate}
\end{corollary}

In view of their generality, Theorem \ref{c1} and Corollary \ref{cc1} have a wide range of applications to boundary value problems. We shall deal with two of them in Subsection 2.4.

\begin{remark}\strut
\begin{enumerate}
\item If $X$ is a function space (which is the case in our applications) and $I_1,I_2$ are even, then one can choose $u_{1,c}$ and $v_{1,c}$ to be nonnegative in Corollary \ref{cc1}.
\item We are not able to establish that $\lambda_{n,c}^+ < \lambda_{n,c}^-$ under the conditions of Theorem \ref{c1}. However, some extra assumptions lead to this inequality, see Subsection 2.3 below. These assumptions entail that $\mathcal{N}^+(\lambda_c)$ is bounded, which is the main ingredient in the proof of $\lambda_{n,c}^+ < \lambda_{n,c}^-$. The boundedness of $\mathcal{N}^+(\lambda_c)$ also implies some properties of the sequence $(u_{n,c})$, see e.g. Corollary \ref{ap1} below.
\end{enumerate}	

\end{remark}

\subsection{The concave-convex problem for the affine $p$-Laplacian}

Our second motivation is related to the problem
\begin{equation} \label{P}
\left\{
\begin{array}{rlllr}
\Delta^{\cal A}_p u &=& \lambda \vert u \vert^{q-2} u +  \vert u\vert^{r-2}u & {\rm in} & \Omega, \\
u&=&0 & {\rm on} & \partial \Omega,
\end{array}\right.
\end{equation}
where $\lambda$ is a positive parameter and $1<q<p<r<p^*$. Here $\Delta^{\cal A}_p$ is the so-called {\it affine} $p$-Laplace operator (we refer to \cite{HJM} for a basic discussion on this operator).
Weak solutions of \eqref{P} are critical points of the functional $\Phi_{\cal A}^\lambda: W^{1,p}_0(\Omega)\setminus \{0\} \rightarrow \mathbb{R}$ given by
$$\Phi_{\cal A}^\lambda(u) = \frac{1}{p} {\cal E}^p_{p,\Omega}(u) - \frac{\lambda}{q}\|u\|_q^{q} - \frac{1}{r}\|u\|_r^{r}, $$
where ${\mathcal{E}}^p_{p,\Omega}: W^{1,p}_0(\Omega)\setminus \{0\} \rightarrow \mathbb{R}$ is the {\it affine}  $p$-energy on $\Omega$, given by
\begin{equation} \label{dep}
{\mathcal{E}}_{p,\Omega}(u) = \gamma_{N,p} \left( \int_{\mathbb{S}^{N-1}} \| \nabla_\xi u\|_p^{-N}\, d\sigma(\xi)\right)^{-\frac{1}{N}}.
\end{equation}
Here $\gamma_{N,p} = \left( 2 \omega_{N+p-2} \right)^{-1} \left(N \omega_N \omega_{p-1}\right) \left(N \omega_N\right)^{p/N}$, $\nabla_\xi u(x)=\nabla u(x) \cdot \xi$ and $\omega_k$ is the volume of the unit Euclidean ball in $\mathbb{R}^k$.

By Theorem 1 of \cite{LM2} we know that $\Phi_{\cal A}^\lambda \in C(W^{1,p}_0(\Omega))$ and  $\Phi_{\cal A}^\lambda \in C^1(W_0^{1,p}(\Omega) \setminus \{0\})$. 
Theorem \ref{tn} applied to $\Phi_{\cal A}^\lambda$ has the following consequence:

\begin{theorem}\label{taf}
There exists $\Lambda_{\cal A}>0$ such that for any $0<\lambda<\Lambda_{\cal A}$ the problem \eqref{P} has two sequences of pairs of solutions $(\pm u_{\lambda,n}) $ and $(\pm v_{\lambda,n}) $. Moreover $u_{1,\lambda}$ is a ground state solution of \eqref{P} and we can choose $u_{\lambda,1},v_{\lambda,1}>0$. The following energy properties hold:

\begin{enumerate}

\item $\Phi_{\cal A}^\lambda(u_{\lambda,n})<0$ and $\Phi_{\cal A}^\lambda(u_{\lambda,n})\leq \Phi_{\cal A}^\lambda(v_{\lambda,n})$ for every $n$. Moreover $\Phi_{\cal A}^\lambda(u_{\lambda,n}) \to 0^-$ and $ \Phi_{\cal A}^\lambda(v_{\lambda,n}) \to \infty$ as $n\to \infty$.
\item $\Phi_{\cal A}^\lambda(u_{\lambda,1})< \Phi_{\cal A}^\lambda(v_{\lambda,1})$ and there exists $0<\bar{\lambda}<\Lambda_{\cal A}$ such that $\Phi_{\cal A}^\lambda(v_{\lambda,1})>0$ for $0<\lambda<\bar{\lambda}$, $\Phi_{\cal A}^{\bar{\lambda}}(v_{\bar{\lambda},1})=0$, and $\Phi_{\cal A}^\lambda(v_{\lambda,1})<0$ for $\bar{\lambda}<\lambda<\Lambda_{\cal A}$.
\item If $q>p\left(1-\frac{1}{N}\right)$ then $\Phi_{\cal A}^\lambda(u_{\lambda,n})< \Phi_{\cal A}^\lambda(v_{\lambda,n})$ for every $n$, $u_{\lambda,n} \to 0$ in $W^{1,p}_0(\Omega)$ as $n\to \infty$, and $\sup_{n \in \mathbb{N}} \|u_{\lambda,n}\| \to 0$ as $\lambda \to 0$.
\end{enumerate}
\end{theorem}

The latter result seems to be the first multiplicity result on \eqref{P}. The existence of a mountain-pass solution (having positive energy) of \eqref{P} has been established for $\lambda>0$ small enough in \cite{LM3}. To this end the authors consider a perturbed functional, as the validity of the Palais-Smale condition for the functional  $\Phi_{\cal  A}$ is still unknown. In the same way as in \cite{LRQS}, we overcome the perturbation argument by dealing with the Palais-Smale condition restricted to $\mathcal{N}(\Phi_{\cal  A})$. Let us note that Theorem \ref{taf} extends some well-known results for the corresponding problem involving the $p$-Laplacian operator, cf. \cite{ABC,BW,GP,I2}.

The rest of this article is organized as follows: in Section 2 we prove Theorems \ref{tn} and \ref{c1}. We also discuss the inequality $\lambda_{n,c}^+< \lambda_{n,c}^-$ related to Theorem \ref{c1}, and apply this theorem to two boundary value problems. In Section 3 we prove Theorem \ref{taf}. Finally, some auxiliary results are proved in Appendix A.

\subsection*{Notation} Throughout this article, we use the following notation:

\begin{itemize}
	\item Unless otherwise stated $\Omega$ denotes a bounded domain of $\mathbb{R}^N$ with $N\geq 1$.
	
	\item Given $r>1$, we denote by $\Vert\cdot\Vert_{r}$ (or $\Vert\cdot\Vert_{r,\Omega}$ in case we need to stress the dependence on $\Omega$) the usual norm in
	$L^{r}(\Omega)$, and by $r^*$ the critical Sobolev exponent, i.e. $r^*=\frac{Nr}{N-r}$ if $r<N$ and $r^*=\infty$ if $r \geq N$.
	
	\item Strong and weak convergences are denoted by $\rightarrow$ and
	$\rightharpoonup$, respectively.
	
	\item Given $g\in L^{1}(\Omega)$ we simply write $\int_{\Omega}g$ instead of $\int_{\Omega} g(x)\, dx$.

\end{itemize}

\section{Proofs of main results}

\subsection{Proof of Theorem \ref{tn}}

Theorem \ref{tn} relies on the homeomorphism between $\mathcal{N}^{\pm}$ and $\mathcal{S}$, which follows from the following result:
\begin{lemma}\label{l1}
Under (H1) the maps $u \mapsto t^{\pm}(u)$ are continuous in $\mathcal{S}$.
\end{lemma}

\begin{proof}
We prove that $u \mapsto t^+(u)$ is continuous in $\mathcal{S}$, the proof being similar for $u \mapsto t^-(u)$. Let $(u_n) \subset \mathcal{S}$ with $u_n \to u$. By (H1) we know that $(t^+(u_n))$ is bounded and away from zero, so we can assume that $t^+(u_n) \to t_0>0$. Thus $\Phi'(t_0u)u=\lim \Phi'(t^+(u_n)u_n)u_n=0$, and it follows that either $t_0=t^+(u)$ or $t_0=t^-(u)$. In addition, if $\delta>0$ is small enough then $\Phi'((t^+(u)-\delta)u)u<0<\Phi'((t^+(u)+\delta)u)u$, so that for $n$ large enough we have $\Phi'((t^+(u)-\delta)u_n)u_n<0<\Phi'((t^+(u)+\delta)u_n)u_n$. Hence $|t^+(u_n)-t^+(u)|< \delta$ for $n$ large enough, i.e. $t_0=t^+(u)$.
\end{proof}

\begin{remark}
More generally, the previous proof shows that  the maps $u \mapsto t^{\pm}(u)$ are continuous in $X \setminus \{0\}$ under (H1).
\end{remark}

\begin{proof}[Proof of Theorem \ref{tn}]
By Lemma \ref{l1}  the map $m^{\pm}:\mathcal{S} \to \mathcal{N}^{\pm}$, $m^{\pm}(u):= t^{\pm}(u)u$ is a homeomorphism. Setting $\hat{\Psi}^{\pm}(u)=\Phi(t^{\pm}(u)u)$ for $u \in X \setminus \{0\}$ one can repeat the proof of \cite[Proposition 9]{SW} to show that $\hat{\Psi}^{\pm} \in C^1(X \setminus \{0\})$.  Note that this proof also holds for local minimisers or local maximisers. Writing $\Psi^{\pm}:=\hat{\Psi}^{\pm}|_S$ we have that $\Psi^{\pm} \in C^1(\mathcal{S})$ and $(\Psi^{\pm})'(u)=\| m^{\pm}(u)\| \Phi'(m^{\pm}(u))$ i.e. $u$ is a critical point of $\Psi^{\pm}$ if, and only if, $m^{\pm}(u)$ is a critical point of $\Phi$. Furthermore:
\begin{itemize}
\item $(m^{\pm}(u_n))$ is a Palais-Smale sequence for $\Phi$ whenever $(u_n)$ is a Palais-Smale sequence for $\Psi^{\pm}$ (with in addition $(t^+(u_n))$ bounded away from zero in case of $\Psi^+$).
\item Any bounded Palais-smale sequence of $\Phi$ in $\mathcal{N}^{\pm}$ yields a Palais-Smale sequence of $\Psi^{\pm}$.
\end{itemize}
The Ljusternik-Schnirelman principle applied to $\Psi^{\pm}$ then shows that $\lambda_n^{\pm}$ are critical values of $\Phi$. The unboundedness of $(\lambda_n^\pm)$ follows from the fact that
the set $K_c:=\{u \in X \setminus \{0\}: \Phi'(u)=0, \Phi(u)=c\}$ is compact for every $c>0$, by the Palais-Smale condition. A standard argument involving the Krasnoselskii genus and a deformation lemma yields the desired conclusion, see e.g. the proof of \cite[Theorem 2.3(ii)]{A} for a similar argument. Finally, $\Phi(t^+(u)u)<\Phi(t^-(u)u)$ for any $u \in X \setminus \{0\}$ implies that $\lambda_n^+ \le \lambda_n^-$ for any $n \ge 1$, this inequality being strict for $n=1$ since $\lambda_1^{\pm}=\displaystyle \min_{u \in \mathcal{S}} \Phi(t^{\pm}(u)u)$.
\end{proof}


\subsection{Proof of Theorem \ref{c1}}

In the sequel we assume the conditions of Theorem \ref{c1}. The proof of this result is a direct consequence of Theorem \ref{tn} and the next results:
\begin{proposition} \label{p2}
 $\lambda_c$ satisfies (H1) and $\mathcal{N}(\lambda_c)$ is bounded away from zero, i.e. $\displaystyle \inf_{u \in \mathcal{N}(\lambda_c)} \|u\|>0$.
\end{proposition}

\begin{proof}
	Indeed, we know that the map $t \mapsto \lambda_c(tu)$ has a critical point if, and only if, $H(tu)=-\alpha c$.
	Since $H(0)=0< -\alpha c$, we see that under $(F1)$  the map $t \mapsto \lambda(tu)$ has two critical points if, and only if, $-\alpha c< \displaystyle \max_{t>0} H(tu)$. Thus $0<-\alpha  c<h_0$ implies that for every $u \in X \setminus \{0\}$ the map $t \mapsto \lambda_c(tu)$ has two critical points $t_c^+(u)<t_c^-(u)$, with $t_c^+(u)$ a local minimum point and $t_c^-(u)$ a local maximum point. We also see from $H(t_c^{\pm}(u)u)=-\alpha c$ and the behavior of $H$ at zero and infinity that $u \mapsto t_c^{\pm}(u)$ are bounded away from zero on $\mathcal{S}$ and bounded from above in any compact subset of $\mathcal{S}$. 
\end{proof}


	

\begin{proposition}\label{p3}
There holds $\lambda_c(u_n) \to \infty$ if $(u_n) \subset \mathcal{N}(\lambda_c)$ and either $u_n \rightharpoonup 0$ or $\|u_n\| \to \infty$.
In particular $\lambda_c$ is coercive and bounded from below on $\mathcal{N}(\lambda_c)$, and satisfies the Palais-Smale condition therein.
\end{proposition}

\begin{proof}
Let $(u_n) \subset \mathcal{N}(\lambda_c)$. By Proposition \ref{p2} there exists $C>0$ such that $\|u_n\|\geq C$ for every $n$. 
\begin{itemize}
\item If $u_n \rightharpoonup 0$ then by (F2) we have that $K(u_n) \to 0$ and $J(u_n) \geq C$ for some $C>0$ and every $n$. Since $I_2(u_n) \to 0$ and $c<0$ we deduce that $\lambda_c(u_n) \to \infty$.\\

\item If $\|u_n\| \to \infty$ then, since $\displaystyle \lim_{t \to \infty} H(tu)=-\infty$ uniformly on weakly compact subsets of $X \setminus \{0\}$ and $H(u_n)=-\alpha c$ we infer that  $v_n:=\frac{u_n}{\|u_n\|} \rightharpoonup 0$. Writing $\lambda_0(u):=\frac{I_1(u)}{I_2(u)}$ we have that $\frac{d}{dt}\lambda_0(tu)=\frac{H(tu)}{t^{\alpha+1}I_2(u)}$. Hence from (H1) we see that for any $u \in X \setminus \{0\}$ the map $t \mapsto \lambda_0(tu)$ has a unique critical point $t_0(u)>0$. Moreover, since $H(tu)>0$ for $0<t<t_c^+(u)$
  we have that $t \mapsto \lambda_0(tu)$ is increasing in $(0,t^+(u))$. In addition, we know that there exists $\delta>0$ such that $t^+(v_n) \geq \delta$ for every $n$. It follows that
$$\lambda_c(u_n)= \lambda_c(t^+(v_n)v_n)\geq \lambda_0(t^+(v_n)v_n) \geq \lambda_0(\delta v_n)\geq \frac{\frac{1}{2}\delta^{\beta-\alpha}-\frac{K(\delta v_n)}{\delta^{\alpha}}}{I_2(v_n)}.$$
Since $v_n \rightharpoonup 0$ in $X$ we have $K(\delta v_n), I_2(v_n) \to 0$, so that $\lambda_c(u_n) \to \infty$. 
\end{itemize}
Finally, if $(u_n) \subset \mathcal{N}(\lambda_c)$ is a Palais-Smale sequence for $\lambda_c$ then by the previous results we know that $(u_n)$ is bounded, so that up to a subsequence $u_n \rightharpoonup u \neq 0$ in $X$. The proof that $u_n \to u$ in $X$ can now be carried out as in \cite[Lemma 3.5]{LRQS}.
\end{proof}




\subsection{On the inequality $\lambda_{n,c}^+< \lambda_{n,c}^-$} \label{ineq}
\medskip

Throughout this subsection we assume the conditions of Theorem \ref{c1}. We shall prove that under some additional conditions (which are satisfied in our applications of Theorem \ref{c1}) we have $\lambda_{n,c}^+< \lambda_{n,c}^-$ for every $n$ and $c<0$ such that $0<-\alpha  c<h_0$. To this end we assume that 
\begin{equation*}
	J(u)=\frac{1}{\beta}\|u\|^\beta+A(u),\ u\in X,
\end{equation*}
where $A$ is bounded on bounded sets, non-negative and $\beta'$-homogeneous for some $\beta'\ge \alpha$. Let us set $P(u):=J'(u)u-\alpha J(u) $ and $Q(u):=\alpha K(u)-K'(u)u$ for $u\in X$, and assume that there exists $p>\beta$ and $C>0$ such that $Q(u)\ge -C \|u\|^p$ for every $u\in X$.  We also assume that $H \in C^1(X)$.

The main property needed in the proof of the inequality $\lambda_{n,c}^+< \lambda_{n,c}^-$ is the boundedness of $\mathcal{N}^+(\lambda_c)$, which we prove next. We also show that this set shrinks to zero as $c \to 0^-$, and $\mathcal{N}^-(\lambda_c)$ is away from zero in the weak topology of $X$:
\begin{lemma}\label{k0} \strut
\begin{enumerate}
\item $\mathcal{N}^+(\lambda_c)$ is bounded and $\displaystyle \sup_{u \in \mathcal{N}^+(\lambda_c)}\|u\| \to 0$ as  $c\to 0^-$.
\item $\mathcal{N}^-(\lambda_c)$ is weakly bounded away from zero, i.e. there is no $(u_n) \subset \mathcal{N}^-(\lambda_c)$ such that $u_n \rightharpoonup 0$ in $X$.
\end{enumerate}	

\end{lemma}

\begin{proof}\strut
	\begin{enumerate}
\item	Indeed, note that
	\begin{equation}\label{Pexpression}
		P(u)=\frac{\beta-\alpha}{\beta}\|u\|^\beta+(\beta'-\alpha)A(u),\ u\in X.
	\end{equation}
	Thus it is clear that $\displaystyle \lim_{t\to \infty}P(tv)=\infty$ uniformly for $v\in \mathcal{S}$. Given $D>-\alpha c$, we choose $t'>0$ such that $P(t'v)>D$ for any $v\in \mathcal{S}$, so that 
	\begin{equation*}
	H(t'v)>D+Q(t'v), \quad \mbox{for any } v\in S.
	\end{equation*}
	Take $\varepsilon>0$ such that $D-\varepsilon>-\alpha c$, and set $V_\varepsilon=\{v\in \mathcal{S}: \left|Q(t'v)\right|< \varepsilon\}$ and $W_\varepsilon= \mathcal{S}\setminus V_\varepsilon$. Note that $v \mapsto t_c^+(v)$ is bounded on $W_\varepsilon$ since, on the contrary, we would find a sequence $(v_n) \subset W_\varepsilon$ such that $t_n:=t_c^+(v_n)\to \infty$ and $H(t_nv_n)=-\alpha c$ for all $n$. Therefore $v_n \rightharpoonup 0$ and $Q(t'v_n)\to 0$, a contradiction. 
	
	Now we claim that $v \mapsto t_c^+(v)$ is also bounded on $V_\varepsilon$. Indeed, note that on $V_\varepsilon$ we have
	\begin{equation*}
	H(t'v)>D-\varepsilon>-\alpha c, 
	\end{equation*}
	which implies that $t_c^+(v)<t'$ for all $v\in  V_\varepsilon$. Thus $v \mapsto t_c^+(v)$ is bounded in $\mathcal{S}$, and the first claim is proved.
   Now note that there exist $C_1,C_2>0$ such that
	\begin{equation*}
		H(u)\ge C_1\|u\|^\beta-C_2\|u\|^p,\ u\in X.
	\end{equation*}
	Consider the function  $t\mapsto m(t)=C_1t^\beta-C_2t^p$, $t>0$, and note that it has a unique global maximizer at $\delta_0$. It follows that $m$ is increasing in $[0,\delta_0]$, $m(0)=0$ and $m(t)>0$ for all $t\in (0,\delta_0]$. Therefore
	\begin{equation*}
		H(v)\ge m(\delta),\ \|v\|=\delta.
	\end{equation*}
	Now given $c<0$ such that $-\alpha c<m(\delta_0)$ it follows that there exists $\delta(c)\in (0,\delta_0)$ such that $m(\delta(c))>-\alpha c$ and hence 	$H(v)\ge m(\delta)>-\alpha c$ whenever $\|v\|=\delta$, which implies that $t_c^+(v)< \delta(c)$ for all $v\in \mathcal{S}$. It is also clear that $\delta(c)$ can be chosen in such a way that $\delta(c) \to 0$ as $c\to 0^-$, which completes the proof.
\item Since $\mathcal{N}^-(\lambda_c)$ is bounded away from zero it suffices to show that $u_n \to 0$ in $X$ if $(u_n) \subset \mathcal{N}^-(\lambda_c)$ and $u_n \rightharpoonup 0$ in $X$. Indeed, in this case since $H'(u_n)u_n<0$ for every $n$ we have that
$$	0 \le \liminf (\beta-\alpha)\|u_n\|^\beta \le\liminf  P'(u_n)u_n\le \liminf -Q'(u_n)u_n=0,$$
and we obtain the desired conclusion.
\end{enumerate}	
\end{proof}

\begin{lemma}\label{k1} If $-\alpha c<h_0$ then there exists $\varepsilon>0$ such that $s(v)-t_c^+(v)\ge \varepsilon$ for all $v\in S$.
\end{lemma}
\begin{proof} It is clear from \eqref{Pexpression} that given $[a,b]\subset (0,\infty)$ and $\varepsilon>0$, there exists $\delta>0$ such that $|P(tv)-P(sv)|<\varepsilon$ for all $t,s\in [a,b]$ with $|t-s|<\delta$, and $v\in \mathcal{S}$.
	
Now	suppose on the contrary that there exists a sequence $(v_n) \subset \mathcal{S}$ such that $s(v_n)-t_c^+(v_n)\to 0$. Since $t_c^+(v_n)$ is bounded  (away from zero and from above), we can assume that  $t_c^+(v_n),s(v_n)\to t>0$.  We can also suppose that $v_n\rightharpoonup v$ in $X$. It follows that $|P(t_c^+(v_n)v_n)-P(s(v_n)v_n)|\to 0$ and $|Q(t_c^+(v_n)v_n)-Q(s(v_n)v_n)|\to 0$, therefore $|H(t_c^+(v_n)v_n)-H(s(v_n)v_n)|\to 0$, which is a contradiction since

	\begin{equation}\label{contradineq}
	H(t_c^+(v_n)v_n)=-\alpha c<h_0\le 	H(s(v_n)v_n), \quad \forall n\in \mathbb{N}.
	\end{equation}
	
\end{proof}
\begin{lemma}\label{k3} Under the above conditions assume moreover that 
	\begin{enumerate}
		\item[(A)] For any $\tilde{A}\ge 0$ and $v\in \mathcal{S}$, the equation $[\beta-\alpha]t^{\beta-1}+\beta'(\beta'-\alpha)t^{\beta'-1}\tilde{A}+Q'(tv)v=0$ has a unique positive solution with respect to $t$.
	\end{enumerate}
	If $-\alpha c<h_0$ and $\varepsilon>0$ is such that $t_c^+(v)+2\varepsilon<s(v)$, then there exists $d>0$ such that $H'(tv)v>d$ for all $v\in S$ and $t\in [t_c^+(v),t_c^+(v)+\varepsilon]$.
\end{lemma}
\begin{proof} Indeed, it is clear that $H'(tv)v>0$ for all $v\in \mathcal{S}$ and $t\in [t_c^+(v),t_c^+(v)+\varepsilon]$. Suppose on the contrary that we can find a sequence $(v_n)\subset \mathcal{S}$ and $(t_n)\subset [t_c^+(v_n),t_c^+(v_n)+\varepsilon]$ such that $H'(t_nv_n)v_n\to 0$. By Lemma \ref{k0} and Proposition \ref{p2} we know that $(t_n)$ is bounded and away from zero. So we can assume that $t_n\to t>0$. We can also assume that $v_n \rightharpoonup v$ in $X$. If $v=0$ then
	\begin{equation*}
	0=(\beta-\alpha)t^{\beta-1}\|v\|^\beta \le \liminf (\beta-\alpha)t_n^{\beta-1}\|v_n\|^\beta \le\liminf  P'(t_nv_n)v_n=\liminf -Q'(t_nv_n)v_n=0,
	\end{equation*}
	which gives a contradiction. Thus $v \neq 0$. Now observe that $H'(s_nv_n)v_n$=0 for all $n$, where $s_n=s(v_n)$, and since $v\neq 0$ it follows by (F1) that $s_n\to s>0$. Since $A(v_n)$ is bounded we have that $A(v_n)\to \tilde{A}\ge 0$ Therefore we obtain
	\begin{equation*}
			[\beta-\alpha]t^{\beta-1}+\beta'(\beta'-\alpha)t^{\beta'-1}\tilde{A}+Q'(tv)v=0=	[\beta-\alpha]s^{\beta-1}+\beta'(\beta'-\alpha)s^{\beta'-1}\tilde{A}+Q'(sv)v,
	\end{equation*}
	which implies, by $(A)$, that $t=s$. However this contradicts Lemma \ref{k1} and the proof is complete.
\end{proof}
\begin{proposition}\label{k4} Assume the conditions of Lemma \ref{k3}. If $-\alpha c<h_0$ then there exists $\delta>0$ such that $\lambda_c(t_c^+(v)v)+\delta<\lambda_c(s(v)v)$ for any $v\in \mathcal{S}$.
\end{proposition}
\begin{proof} 	Note that for any $v\in \mathcal{S}$ we have
	\begin{equation*}
	\lambda_c(s(v)v)-\lambda_c(t_c^+(v)v)=\int_{t_c^+(v)}^{s(v)}\frac{H(tv)+\alpha c}{t^{\alpha+1}I_2(v)}\, dt, 
	\end{equation*}
	and
	\begin{equation*}
	H(tv)+\alpha c=	\int_{t_c^+(v)}^tH'(sv)v \, ds.
	\end{equation*}
	By Lemmas \ref{k1} and \ref{k3} we can find $\varepsilon>0$ such that $t_c^+(v)+\varepsilon<s(v)$ and $H'(tv)v>d>0$ for all $t\in [t_c^+(v),t_c^+(v)+\varepsilon]$. Thus
	$$\lambda_c(s(v)v)-\lambda_c(t_c^+(v)v)>\int_{t_c^+(v)}^{t_c^+(v)+\varepsilon}\frac{H(tv)+\alpha c}{t^{\alpha+1}I_2(v)}\, dt 
		> D\int_{t_c^+(v)}^{t_c^+(v)+\varepsilon}(t-t_c^+(v))dt 
		= D \frac{\varepsilon^2}{2}.$$
\end{proof}

As a consequence of the previous results we have the following inequalities, where $\displaystyle
\lambda_{n,c}:=\inf_{F\in \mathcal{F}_n}\sup_{v\in F}\lambda_c(s(v)v)$:
\begin{theorem} Under the above conditions there holds $\lambda_{n,c}^+<\lambda_{n,c}\le \lambda_{n,c}^-$ for every $n$ and $c<0$ such that $0<-\alpha  c<h_0$. 
\end{theorem}

\subsection{Applications of Corollary \ref{cc1} to some boundary value problems}

First we consider the problem \begin{equation}\label{cc}
\begin{cases}
-\Delta u=\lambda |u|^{q-2}u +f(u) &\mbox{ in } \Omega,\\
u=0 &\mbox{ on } \partial \Omega,
\end{cases}
\end{equation}
with $1<q<2$ and $f\in C^1(\mathbb{R})$ being subcritical and superlinear, in which case the problem has a concave-convex nature.

We have  $I_1(u)=\frac{1}{2}\int_\Omega |\nabla u|^2-\int_\Omega F(u)$ and $I_2(u)=\frac{1}{q}\int_\Omega |u|^q$ for $u \in H_0^1(\Omega)$, so that
\begin{equation}\label{lex}
\lambda_c(u)=\frac{\frac{q}{2}\int_\Omega |\nabla u|^2-q\int_\Omega F(u)-qc}{\int_\Omega |u|^q},
\end{equation}
for $u \in H_0^1(\Omega) \setminus \{0\}$.
Thus $I_2$ is now $q$-homogeneous, and
$$H(u)=\frac{2-q}{2}\int_\Omega |\nabla u|^2+\int_\Omega \left(qF(u)-f(u)u\right).$$

We shall assume that 

\begin{enumerate}
	\item[(f1)] 	$|f(s)| \leq C(1+|s|^{r-1})$ for some $C>0$ and $1<r<2^*$, and every $s \in \mathbb{R}$.
	\item [(f2)] $t\mapsto \frac{(q-1)f(t)}{t}-f'(t)$ is decreasing in $(0,\infty)$ and increasing in $(-\infty,0)$, and $$\displaystyle \lim_{|t|\to \infty} \left[\frac{(q-1)f(t)}{t}-f'(t)\right]=-\infty.$$
\end{enumerate} 

These conditions imply that $H$ has a positive ground state level $h_0>0$,  given by \eqref{cal}.
\begin{corollary}\label{ap1}
	Let   $1<q<2$, and $f\in C^1(\mathbb{R})$ with $f'(0)=f(0)=0$ and satisfying (f1) and (f2). Then for every $c<0$ such that $-\alpha c<h_0$ the following assertions hold:
	\begin{enumerate}
		\item The problem \eqref{cc} has a  nontrivial solution $ u_{1,c}$ having energy $c$ for $\lambda=\lambda_{1,c}^+$, and has no such solution for $\lambda<\lambda_{1,c}^+$.
		\item If in addition $f$ is odd, then we can choose $u_{1,c}$ to be nonnegative, and for every $n\in \mathbb{N}$ the problem \eqref{cc} has a pair of nontrivial solutions $\pm u_{n,c}$ having energy $c$ for $\lambda=\lambda_{n,c}^+$ and a pair of nontrivial solutions $\pm v_{n,c}$ having energy $c$ for $\lambda=\lambda_{n,c}^-$. Moreover:
		\begin{enumerate}
\item $\lambda_{n,c}^+<\lambda_{n,c}^-$ for every $n$, and $\lambda_{n,c}^{\pm} \nearrow \infty$ as $n \to \infty$.
\item $u_{n,c} \rightharpoonup 0$ with $u_{n,c} \not \to 0$ in $X$, and $\|v_{n,c}\| \to \infty$ as $n \to \infty$.
\item $\sup_n \|u_{n,c}\| \to 0$ as $c \to 0^-$.
		\end{enumerate}
		  
	\end{enumerate}
\end{corollary}

\begin{proof}
	It is a consequence of Corollary \ref{cc1} with $I_1=J-K$ with $J(u):=\frac{1}{2}\int_\Omega |\nabla u|^2$ and $K(u):=\int_\Omega F(u)$, and $I_2(u)=\frac{1}{q} \|u\|_q^q$ for $u \in H_0^1(\Omega)$. Following the arguments in the proof of \cite[Corollary 4.5]{LRQS}  one can easily check that (F1) and (F2) are satisfied, and $h_0>0$. One can also check that the conditions of Subsection 2.2 are satisfied. Note that $I_1$ is bounded on bounded sets, and since $\lambda_c(u_{n,c}), \lambda_c(v_{n,c}) \to \infty$ we have either $w_n \rightharpoonup 0$ in $X$ or $\|w_n\| \to \infty$ for $w_n=u_{n,c}$ and $w_n=v_{n,c}$. Lemma \ref{k0} yields that $u_{n,c} \rightharpoonup 0$ in $X$ and $\|v_{n,c} \|\to \infty$, whereas Proposition \ref{p2} shows that $u_{n,c} \not \to 0$ in $X$.	
\end{proof}

Corollary \ref{ap1} extends most results of \cite[Theorem 2.3]{RSiS}, which holds for the $p$-Laplacian version of \eqref{cc}, to a nonpowerlike $f$. It is not difficult to see that Corollary \ref{ap1} can be easily extended to the $p$-Laplacian setting, in the same spirit of \cite[Remark 4.4]{LRQS}.

As a second application of Corollary \ref{cc1} we deal with the problem
\begin{equation}\label{ccpq}
\begin{cases}
-\Delta_p u - \Delta_q u=\lambda |u|^{r_1-2}u + |u|^{r_2-2}u&\mbox{ in } \Omega,\\
u=0 &\mbox{ on } \partial \Omega,
\end{cases}
\end{equation}
where $1<r_1<q<p<r_2<p^*$. This problem can be seen as an extension of \eqref{cc}  (with $f(u)=|u|^{r_2-2}u$), as the Laplacian operator is replaced by the $(p,q)$-Laplacian operator $-\Delta_p-\Delta_q$.

We have now $I_1(u)=\frac{1}{p}\int_\Omega |\nabla u|^p+\frac{1}{q}\int_\Omega |\nabla u|^q-\frac{1}{r_2}\int_\Omega |u|^{r_2}$ and $I_2(u)=\frac{1}{r_1}\int_\Omega |u|^{r_1}$ for $u \in W_0^{1,p}(\Omega)$, and
\begin{equation}
\lambda_c(u)=\frac{\frac{r_1}{p}\int_\Omega |\nabla u|^p+\frac{r_1}{q}\int_\Omega |\nabla u|^q-\frac{r_1}{r_2}\int_\Omega |u|^{r_2}-cr_1}{\int_\Omega |u|^{r_1}},
\end{equation}
for $u \in W_0^{1,p}(\Omega) \setminus \{0\}$.
Thus $I_2$ is now $r_1$-homogeneous, and
$$H(u)=\frac{p-r_1}{p}\int_\Omega |\nabla u|^p+\frac{q-r_1}{q}\int_\Omega |\nabla u|^q-\frac{r_2-r_1}{r_2}\int_\Omega |u|^{r_2}.$$
Since $q<p<r_2$ we see that (F1) is satisfied and that $H$ has a positive ground state level, i.e. $h_0>0$. Note also that $I_1=J-K$ with $J(u)=\frac{1}{p}\|u\|^p+\frac{1}{q}\int_\Omega |\nabla u|^q$ and $K(u)=\frac{1}{r_2}\|u\|_{r_2}^{r_2}$, which satisfy (F2) and the conditions of Subsection 2.2. Therefore we deduce the following result:

\begin{corollary}
Let $1<r_1<q<p<r_2<p^*$ and $c<0$ be such that $-\alpha c<h_0$. Then the conclusions of Corollary \ref{ap1} apply to \eqref{ccpq}.
\end{corollary}

\section{Proof of Theorem \ref{taf}}

We note that $\Phi_{\cal A}^\lambda$ belongs to the class \eqref{egf} (see Appendix A below) with $E(u)={\cal E}^p_{p,\Omega}(u)$, $A(u)=\|u\|_q^q$, and $B(u)=\|u\|_r^r$. 
From the affine Sobolev inequality it follows that (H4) is satisfied.  
We set $$\Lambda_{\cal A}:=C(p,q,r) \inf_{u \in W_0^{1,p}(\Omega) \setminus \{0\}} \frac{{\cal E}_{p,\Omega}(u)^{p\frac{r-q}{r-p}}}{\|u\|_q^q\|u\|_r^{r\frac{p-q}{r-p}}}.$$
and we write $$\mathcal{N}_{\lambda}^{\pm}:=\mathcal{N}^{\pm}(\Phi_{\cal A}^\lambda).$$

\begin{proposition} \label{pa1}
There holds $\Lambda_{\cal A}>0$ and $\Phi_{\cal A}^\lambda$ satisfies (H1) for $0<\lambda<\Lambda_{\cal A}$. 
\end{proposition}

\begin{proof}
It is a consequence of Corollary \ref{c0} and Lemma \ref{l0}.
\end{proof}

\begin{proposition} \label{pa2}
$\Phi_{\cal A}^\lambda$ satisfies $(H2)^{\pm}$.
\end{proposition}

\begin{proof}
It is clear that $\Phi_{\cal A}^\lambda$ is even and by Lemma \ref{l0} it is bounded from below on $\mathcal{N}_{\lambda}$. Let $(u_k) \subset \mathcal{N}_{\lambda}^{\pm}$ be such that $(\Phi_{\cal A}^\lambda(u_k))$ is bounded and $\Vert (\Phi_{\cal A}^\lambda)'(u_k) \Vert_{W^{-1, p'}(\Omega)} \rightarrow 0$. We claim that, up to a subsequence, $u_k \rightarrow u \neq 0$ in $L^s(\Omega)$, for $1\leq s<p^*$, so that  $(u_k)$ is bounded in $W_0^{1,p}(\Omega)$ by  \cite[Corollary 2.1]{LM1}. Indeed, 
if $(u_k)\subset \mathcal{N}_{\lambda}^-$ then, by Lemma \ref{l0}, we know that $({\cal E}_{p,\Omega}(u_k))$ is bounded (away from zero and from above). So, by the affine Rellich-Kondrachov compactness theorem (see \cite[Theorem 6.5.3]{T}) we can assume that $u_k \rightarrow u$ in $L^s(\Omega)$, for $1\leq s<p^*$, and from 
${\cal E}_{p,\Omega}(u_k)=\lambda \|u_k\|_q^q+\|u_k\|_r^r$
 we find that $u\neq 0$.  If now  $(u_k)\subset \mathcal{N}_{\lambda}^+$ then we can assume that $\Phi_{\cal A}^\lambda(u_k) \to \lambda_n^+$, and since $\Phi_{\cal A}^\lambda(u)<0$ for any $u \in \mathcal{N}_{\lambda}^+$ we have $\lambda_n^+<0$ for every $n$. We know that
\[
\lambda\left(\frac{1}{q}-\frac{1}{r}\right)\|u_k\|_q^q =  \left(\frac{1}{p}-\frac{1}{r}\right){\cal E}^p_{p,\Omega}(u_k) - \Phi_{\cal A}^\lambda(u_n),
\]
so that  $\|u\|_q > 0$ i.e. $u\neq 0$. Thus we can assume that $u_k \rightharpoonup u \neq 0$ in $W_0^{1,p}(\Omega)$. Moreover, $\Vert \Phi'_{\cal A}(u_k) \Vert_{W^{-1, p'}(\Omega)} \rightarrow 0$ clearly implies that $\Phi'_{\cal A}(u_k)(u_k - u) \to 0$,
which by standard arguments yields
$ \langle \Delta^{\cal A}_p u_k, u_k - u \rangle \to 0$.
Finally, by \cite[Theorem 2]{LM2}, $u_k \to u$ in $W_0^{1,p}(\Omega)$, and the proof is complete.\\
\end{proof}


\begin{proposition}\label{pa3}
Assume that $p>q>\left(1-\frac{1}{N}\right)p$. Then there exists a constant $C>0$ such that $\|u\| \leq C  \lambda^{\frac{p}{p-q}}$ for any $u \in \mathcal{N}_{\lambda}^{+}$. In particular $\mathcal{N}_{\lambda}^{+}$ is bounded.
\end{proposition}

\begin{proof}

Indeed, let $u \in \mathcal{N}_{\lambda}^{+}$. Then ${\cal E}_{p,\Omega}(u)< C\lambda \Vert u\Vert_{q}^{\frac{q}{p}}$. On the other hand, by  \cite[Theorem 9]{HJM}, we get
${\cal E}_{p,\Omega}(u)\geq C\Vert u\Vert_{p}^{\frac{N-1}{N}}\Vert u\Vert^{\frac{1}{N}}$,
where $C$ is a positive constant. Then
$\Vert u\Vert^{\frac{1}{N}}\leq C\lambda^N \Vert u\Vert_{q}^{\frac{q}{p}}\Vert u\Vert_{p}^{-\frac{N-1}{N}}$
and since $q<p$, we have $\Vert u\Vert\leq C\lambda \Vert u\Vert_{p}^{N(\frac{q}{p}-1)+1}$. Now, note that $q>\left(1-\frac{1}{N}\right)p$ implies $0<N\left(\frac{q}{p}-1\right)+1<1$. Therefore
$\Vert u\Vert\leq C\lambda^N \Vert u\Vert^{N(\frac{q}{p}-1)+1}$ and so $\Vert u\Vert\leq C \lambda^{\frac{p}{p-q}}$.

\end{proof}

\begin{remark} Denote $P(u)={\cal E}^p_{p,\Omega}(u)$, $Q_\lambda(u)=-\lambda\|u\|_q^q-\|u\|_r^r$, and  $H_\lambda(u):=(\Phi_{\cal A}^\lambda)'(u)u =P(u)+Q_\lambda(u)$. Then the condition $\lambda\in(0,\Lambda_{\mathcal{A}})$ is equivalent to the following fact: for every $v\in \mathcal{S}$, the map $(0,\infty)\ni t\mapsto H_\lambda(tv)$ has exactly two zeros $t^+(v)<t^-(v)$, and a unique local maximizer  $s(v)\in (t^+(v),t^-(v))$.  From the definition of $\Lambda_{\mathcal{A}}$ it follows that
$\displaystyle \inf_{v\in \mathcal{S}} H_\lambda(s(v)v)=i_\lambda>0$ for $\lambda\in(0,\Lambda_{\mathcal{A}})$.
Assume now that $\mathcal{N}_{\lambda}^{+}$ is bounded. Then, by adapting the proofs of Lemmas \ref{k1}, \ref{k3} and Proposition \ref{k4} in Sections \ref{ineq}, we can prove the following:
\begin{lemma}\label{k11} If $\lambda\in(0,\Lambda_{\mathcal{A}})$ then there exists $\varepsilon>0$ such that $s(v)-t^+(v)\ge \varepsilon$ for all $v\in S$.
\end{lemma}
In the proof of Lemma \ref{k11} we see that the contradiction obtained in inequality \ref{contradineq} is achieved by noting that now we have
$H_\lambda(t^+(v_n)v_n)=0<i_\lambda\le 	H_\lambda(s(v_n)v_n)$  for every $n\in \mathbb{N}$.

\begin{lemma}	If $\lambda\in(0,\Lambda_{\mathcal{A}})$ and $\varepsilon>0$ is such that $t^+(v)+2\varepsilon<s(v)$, then there exists $d>0$ such that $H_\lambda'(tv)v>d$ for all $v\in S$ and $t\in [t^+(v),t^+(v)+\varepsilon]$.
\end{lemma}
\begin{proposition} If $\lambda\in(0,\Lambda_{\mathcal{A}})$ then there exists $\delta>0$ such that $\Phi_{\cal A}^\lambda(t^+(v)v)+\delta<\Phi_{\cal A}^\lambda(s(v)v)$ for any $v\in \mathcal{S}$.
\end{proposition}
Combining the previous results with Proposition \ref{pa3} we derive the following:
\begin{theorem}\label{ta3} Assume that $p>q>\left(1-\frac{1}{N}\right)p$ and $\lambda\in(0,\Lambda_{\mathcal{A}})$. Then $\Phi_{\cal A}^\lambda(u_{\lambda,n})< \Phi_{\cal A}^\lambda(v_{\lambda,n})$ for every $n$.
\end{theorem}
\end{remark}

\begin{proof}[Proof of Theorem \ref{taf}]
	By Propositions \ref{pa1} and \ref{pa2}, and Theorem \ref{tn}, we know that 
	$\displaystyle \inf_{F\in \mathcal{F}_n}\sup_{u\in F}\Phi_{\cal A}^\lambda(t^{\pm}(u)u)$ are two nondecreasing sequences of critical values of $\Phi_{\cal A}^\lambda$ for $0<\lambda<\Lambda_{\cal A}$. Thus for such values of $\lambda$ there exist two sequences $(u_{\lambda,n}),(v_{\lambda,n})$ of critical points of $\Phi_{\cal A}^\lambda$  with $(u_{\lambda,n}) \subset \mathcal{N}_\lambda^+$ and $(v_{\lambda,n}) \subset \mathcal{N}_\lambda^-$.
	Let us prove the energy properties of $(u_{\lambda,n})$ and $(v_{\lambda,n})$:
	\begin{enumerate}
\item  It is clear that $\Phi_{\cal A}^\lambda(u_{\lambda,n})\leq \Phi_{\cal A}^\lambda(v_{\lambda,n})$ and, since $\Phi_{\cal A}^\lambda<0$ in $\mathcal{N}_\lambda^+$, that $\Phi_{\cal A}^\lambda(u_{\lambda,n})<0$ for every $n$.
Moreover, $\Phi_{\cal A}^\lambda$ satisfies the Palais-Smale condition on $\mathcal{N}_\lambda^-$ at any level, so $\Phi_{\cal A}^\lambda(v_{\lambda,n}) \to \infty$. We also see that $\Phi_{\cal A}^\lambda(u_{\lambda,n}) \to 0^-$ as $n\to \infty$  by noting that $-\left(\Phi_{\cal A}^\lambda\right)^{-1}$ satisfies the Palais-Smale condition at any positive level (since $\Phi_{\cal A}^\lambda$ does so at any negative level). 
\item The inequality $\Phi_{\cal A}^\lambda(u_{\lambda,1})<\Phi_{\cal A}^\lambda(v_{\lambda,1})$ follows from the fact that these values correspond to $\min_{\mathcal{N}_\lambda^+}  \Phi_{\cal A}^\lambda$ and $\min_{\mathcal{N}_\lambda^-}  \Phi_{\cal A}^\lambda$, respectively. In addition, it is clear that $\Phi_{\cal A}^\lambda(v_{\lambda,1})>0$ for $\lambda$ close to $0$ and $\Phi_{\cal A}^\lambda(v_{\lambda,1})<0$ for $\lambda$ close to $\Lambda_{\cal{A}}$. The continuous and decreasing behavior of $\lambda \mapsto \Phi_{\cal A}^\lambda(v_{\lambda,1})$ completes the picture. 
	\item  Proposition \ref{pa3} yields that $\sup_{n \in \mathbb{N}} \|u_{\lambda,n}\| \to 0$ as $\lambda \to 0$, whereas Theorem \ref{ta3} yields that $\Phi_{\cal A}^\lambda(u_{\lambda,n})< \Phi_{\cal A}^\lambda(v_{\lambda,n})$ for every $n$ and $0<\lambda<\Lambda_{\cal A}$. Finally, since $\mathcal{N}_\lambda^+$ is bounded we know that $(u_{\lambda,n})$ is bounded, so we may assume that $u_{\lambda,n} \rightharpoonup u_\lambda$ in $W^{1,p}_0(\Omega)$. If $u_\lambda \not \equiv 0$ then we have $u_{\lambda,n} \to u_\lambda$ in $W^{1,p}_0(\Omega)$, and consequently $\left(\Phi_{\cal A}^\lambda\right)'(u_\lambda)=0$ and $\Phi_{\cal A}^\lambda(u_\lambda)=0$, which contradicts $u_\lambda \in \mathcal{N}_\lambda^+$. Thus $u_\lambda \equiv 0$. From the proof of Proposition \ref{pa3} we know that $\Vert u\Vert\leq C\lambda \Vert u\Vert_{p}^{N(\frac{q}{p}-1)+1}$, and therefore $u_{\lambda,n} \to 0$ in $W^{1,p}_0(\Omega)$.
\end{enumerate}	
\end{proof}
\medskip
\appendix
\medskip

\section{A class of functionals with homogeneous terms}

In the sequel we consider 
\begin{equation}\label{egf}
\Phi_\lambda(u):=\frac{1}{\eta}E(u)-\frac{\lambda}{\alpha}A(u)-\frac{1}{\beta} B(u),
\end{equation} where $\lambda>0$ and
$A,B,E \in C^1(X \setminus \{0\})$ are positive and $\alpha,\beta,\eta$-homogeneous respectively, with $\alpha<\eta<\beta$, $\alpha,\eta,\beta \neq 0$.
In particular $\alpha$, $\eta$ or $\beta$ can be negative. 

\begin{proposition}\label{p1}
	Under the previous conditions, let $u \in X \setminus \{0\}$. Then there exists $\lambda(u)>0$ such that:
	\begin{enumerate}
		\item For $0<\lambda <\lambda(u)$ the map $t \mapsto \Phi_\lambda(tu)$ has exactly two critical points $t^+(u)<t^-(u)$, with $t^+(u)$ a local minimum point and $t^-(u)$ a local maximum point.
		\item For $\lambda=\lambda(u)$ the map $t \mapsto \Phi_\lambda(tu)$ has exactly one critical point, which is degenerate.
		\item For $\lambda>\lambda(u)$ the map $t \mapsto \Phi_\lambda(tu)$ has no critical point.
	\end{enumerate}
\end{proposition}

\begin{proof}
	Let $\varphi_{u,\lambda}(t)=\Phi_\lambda(tu)$. Then $\varphi_{u,\lambda}'(t)=t^{\eta-1} E(u)-\lambda t^{\alpha-1} A(u)-t^{\beta-1} B(u)$. Since $\alpha<\eta<\beta$ and $A(u),B(u)>0$ we see that $\varphi_{u,\lambda}'(t)<0$ for $t$ small enough and for $t$ large enough. So the existence of critical points of $\varphi_{u,\lambda}$ is equivalent to the condition $\displaystyle \max \varphi_{u,\lambda}' \geq 0$. We also see that for $\lambda>0$ small enough we have $\displaystyle \max_{t>0} \varphi_{u,\lambda}'(t) >0$ and since $\lambda \mapsto \varphi_{u,\lambda}'(t)$ is decreasing there exists $\lambda(u)>0$  such that $\displaystyle \max \varphi_{u,\lambda(u)}'=0$. It turns out that $\displaystyle \max_{t>0} \varphi_{u,\lambda(u)}'(t) =\varphi_{u,\lambda(u)}'(t_0(u))=0$ if, and only if, the couple $(\lambda(u),t_0(u))$ solves 
	$\varphi_{u,\lambda(u)}'(t_0(u))=\varphi_{u,\lambda(u)}''(t_0(u))=0$, i.e. $t_0(u)=\left(\frac{\eta-\alpha}{\beta-\alpha} \frac{E(u)}{B(u)}\right)^{\frac{1}{\beta-\eta}}$ and  \begin{equation}\label{dl} \lambda(u):=C(\alpha,\eta,\beta) \frac{E(u)^{\frac{\beta-\alpha}{\beta-\eta}}}{A(u)B(u)^{\frac{\eta-\alpha}{\beta-\eta}}},
	\end{equation}
	where $C(\alpha,\eta,\beta):=\frac{\beta-\eta}{\beta-\alpha} \left(\frac{\eta-\alpha}{\beta-\alpha}\right)^{\frac{\eta-\alpha}{\beta-\eta}}$.
\end{proof}

Next we set $\Lambda:=\displaystyle \inf_{X \setminus \{0\}} \lambda(u)$
and assume that 

\begin{enumerate}
	\item[(H3)] There exists $C>0$ such that  $B(u)\leq C E(u)^{\frac{\beta}{\eta}}$ and $E(u) \leq C\|u\|^{\eta}$ for any $u \in X \setminus \{0\}$.
\end{enumerate}

\begin{corollary}\label{c0}
	Assume (H3). If $0<\lambda<\Lambda$ then $\Phi_\lambda$ satisfies (H1).
\end{corollary}

\begin{proof}
	Let $0<\lambda<\Lambda$. By Proposition \ref{p1} we know that $t \mapsto \Phi_\lambda(tu)$ has exactly two critical points $t^+(u)<t^-(u)$, with $t^+(u)$ a local minimum point and $t^-(u)$ a local maximum point. Moreover, we have 
	$$t^{\pm}(u)^{\eta-\alpha} E(u) -\lambda A(u) - t^{\pm}(u)^{\beta-\alpha}B(u)=0 \quad \forall u \in \mathcal{S},$$ and since $A,B,E$ are continuous and positive on $\mathcal{S}$ we see that $u \mapsto t^{\pm}(u)$ are bounded (away from zero and from above) in any compact subset of $\mathcal{S}$. Finally, we have that $E(t^-(u)u)<\frac{\beta-\eta}{\eta-\alpha} B(t^-(u)u)$, so that
	$$E(u)<\frac{\beta-\eta}{\eta-\alpha} t^-(u)^{\beta-\eta}B(u)<Ct^-(u)^{\beta-\eta}E(u)^{\frac{\beta}{\eta}},$$
	i.e. $Ct^-(u)^{\beta-\eta}\geq E(u)^{\frac{\eta-\beta}{\eta}}$ for some $C>0$ and any $u \in X\setminus \{0\}$. Since $E(u) \leq C\|u\|^{\eta}$ and $\beta>\eta$ we infer that $t^-(u)$ is away from zero on $\mathcal{S}$.
\end{proof}

Let us deal with the following conditions:
\begin{enumerate}
	\item[(H4)] There exists $C>0$ such that $A(u)\leq C E(u)^{\frac{\alpha}{\eta}}$ for every $u \in X \setminus \{0\}$.
	\item[(H5)] There exists $C>0$ such that $A(u)\leq C B(u)^{\frac{\alpha}{\beta}}$ for every $u \in X \setminus \{0\}$.
\end{enumerate}

\begin{lemma}\label{l0}
	Assume (H3) and either that $\alpha>0$ and (H4) holds or $\alpha<0<\eta$ and (H5) holds. Then:
	\begin{enumerate}
		\item  $\Lambda>0$.
		\item $\Phi_\lambda$ is bounded from below on $\mathcal{N}_\lambda$. Moreover, if $(u_n) \subset \mathcal{N}_\lambda$ and $(\Phi_\lambda(u_n))$ is bounded then $(E(u_n))$ is bounded.
		\item $E$ is bounded away from zero on $\mathcal{N}_\lambda^-$.
	\end{enumerate}
\end{lemma}

\begin{proof}\strut
	\begin{enumerate}
		\item Note that $$\frac{E(u)^{\frac{\beta-\alpha}{\beta-\eta}}}{A(u)B(u)^{\frac{\eta-\alpha}{\beta-\eta}}}=\frac{E(u)^{\frac{\alpha}{\eta}}}{A(u)}\left(\frac{E(u)^{\frac{\beta}{\eta}}}{B(u)}\right)^{\frac{\eta-\alpha}{\beta-\eta}}.$$
		If $\alpha>0$ and (H4) holds then using (H3) we see that $\Lambda>0$. Now, if $\alpha<0<\eta$ and (H5) holds then 
		$$\frac{E(u)^{\frac{\beta-\alpha}{\beta-\eta}}}{A(u)B(u)^{\frac{\eta-\alpha}{\beta-\eta}}}\geq \frac{1}{C} \left(\frac{E(u)}{B(u)^{\frac{\eta}{\beta}}}\right)^{\frac{\beta-\alpha}{\beta-\eta}}.$$
		Since $\beta>\eta>0$ we have, from (H3), $E(u) \geq C B(u)^{\frac{\eta}{\beta}}$ for some $C>0$ and any $u \in X\setminus \{0\}$, which yields the desired inequality.\\
		
		\item Indeed, since $\mathcal{N}_\lambda=\{u \in X \setminus \{0\}: E(u)-\lambda A(u) - B(u)=0\}$
		we have $$\Phi_\lambda(u)=\frac{\beta-\eta}{\beta \eta} E(u) -\lambda \frac{\beta-\alpha}{\alpha \beta} A(u) \quad \mbox{for } u \in \mathcal{N}_\lambda.$$
		If $\alpha>0$ and (H4) holds then $\Phi_\lambda(u) \geq C_1 E(u)-C_2 E(u)^{\frac{\alpha}{\eta}}$ for some $C_1,C_2>0$ and any $u \in \mathcal{N}_\lambda$.
		Since $E$ is positive and $\alpha<\eta$ we infer the desired conclusion.
		Now, if $\alpha<0<\eta$ and (H5) holds then $\Phi_\lambda(u)>\frac{\beta-\eta}{\beta \eta} E(u)$, and the conclusion follows.\\
		
		
		\item If $u \in \mathcal{N}_\lambda^-$ then, by (H3), we have that $$E(u)<\frac{\beta-\eta}{\eta-\alpha} B(u)\leq C E(u)^{\frac{\beta}{\eta}},$$
		which yields the conclusion.
	\end{enumerate}
\end{proof}


\begin{lemma}\label{NBH} For any $u\in X\setminus\{0\}$ there holds:
 $$t^+(u)<\left(\lambda\frac{\eta-\alpha}{\eta-\beta}\frac{A(u)}{E(u)}\right)^{\frac{1}{\eta-\alpha}} \quad \mbox{and} \quad
		 t^-(u)>\left(\frac{\beta-\alpha}{\eta-\alpha}\frac{E(u)}{B(u)}\right)^{\frac{1}{\beta-\eta}}.$$
	
\end{lemma}

\begin{proof} Indeed, let $\varphi_{\lambda,u}(t)=\Phi_\lambda(tu)$ for $t>0$ and $u\in X\setminus\{0\}$. If $\varphi_{\lambda,u}'(t)=0<\varphi_{\lambda,u}''(t)$ then
	\begin{equation*}
	(\beta-\eta)t^\eta E(u)-\lambda(\alpha-\eta)t^\alpha A(u)>0, 
	\end{equation*}
	which implies that
	\begin{equation*}
	t^+(u)=t<\left(\lambda\frac{\eta-\alpha}{\eta-\beta}\frac{A(u)}{E(u)}\right)^{\frac{1}{\eta-\alpha}}.
	\end{equation*}
	Now, if $\varphi_{\lambda,u}'(t)=0>\varphi_{\lambda,u}''(t)$ then
	\begin{equation*}
	(\beta-\alpha)t^\eta E(u)-(\eta-\alpha)t^\beta B(u)<0, 
	\end{equation*}
	so
	\begin{equation*}
	t^-(u)=t>\left(\frac{\beta-\alpha}{\eta-\alpha}\frac{E(u)}{B(u)}\right)^{\frac{1}{\beta-\eta}}.
	\end{equation*}
\end{proof}


\medskip
\begin{thebibliography}{99}   
\medskip
	
	
	
	
	\bibitem{A} A. Ambrosetti, {\em Critical points and nonlinear variational problems}, Mém. Soc. Math. Fr. 2e Sér. 49 (1992).
	
	\bibitem{ABC}A. Ambrosetti, H. Brezis, G. Cerami, \textit{Combined effects
		of concave and convex nonlinearities in some elliptic problems}, J. Funct.
	Anal. \textbf{122} (1994), 519-543.
	
	
	
	
	
	\bibitem{BW}  T. Bartsch, M. Willem, \textit{ On an elliptic equation with concave and convex nonlinearities.} Proc. Amer. Math. Soc. 123 (1995), no. 11, 3555–3561.
	
	
	

	\bibitem{BWu}  K. J. Brown, T.F. Wu, {\em A fibering map approach to a potential operator equation and its applications}. Differential Integral Equations 22 (2009), no. 11-12, 1097–1114.

	
	
	

	
	
	

	
	\bibitem{GP} J. Garcia Azorero and I. Peral Alonso, {\em Multiplicity of solutions for elliptic problems with critical exponent or with a non-symmetric term}, Trans. Amer. Math. Soc., 323 (1991), 877-895.
	
	
	

\bibitem{HJM} J. Haddad, C.H. Jiménez, M. Montenegro, {\em
From affine Poincaré inequalities to affine spectral inequalities}, Adv. Math. 386 (2021) 107808, 35pp.
	
	
	
	\bibitem{I2} Y. Il'yasov,
	\newblock {\em On nonlocal existence results for elliptic equations with convex–concave nonlinearities},
	\newblock {Nonlinear Analysis: Theory, Methods \& Applications {\bf 15} (2005), 211--236.}
	
	\bibitem{I3} Y. Il'yasov, {\em Rayleigh quotients of the level set manifolds related to the nonlinear PDE}, Minimax Theory and its Applications, 07 (2) (2022), 277–302.
	
\bibitem{LM1} E. J. F. Leite, M. Montenegro, {\em Least energy solutions for affine $p$-Laplace equations involving subcritical and critical nonlinearities}, Adv. Calculus of Variations, 2024. 

\bibitem{LM2} E. J. F. Leite, M. Montenegro, {\em Towards existence theorems to affine $p$-Laplace equations via variational approach}, Calc. Var. Partial Differential Equations 63, 73 (2024). 

\bibitem{LM3} E. J. F. Leite, M. Montenegro, {\em Nontrivial solutions to affine $p$-Laplace equations via a perturbation strategy}, Nonlinear Analysis: Real World Applications, Volume 80, 2024. 

\bibitem{LRQS} E. J. F. Leite, H. Ramos Quoirin, K. Silva {\em Some applications of the Nehari manifold method to functionals in $C^1(X \setminus \{0\})$}, Calc. Var. 64, 56 (2025).
	


	
	
	
	
	\bibitem{RSiS}H. Ramos Quoirin, G. Siciliano, K. Silva, {\em  Critical points with prescribed energy for a class of functionals depending on a parameter: existence, multiplicity and bifurcation results}, Nonlinearity 37 (2024), no. 6, Paper No. 065010, 40 pp.
	
	
	
	
	
	
	
	

	
	
	
	
	
	\bibitem{SW} A. Szulkin and T. Weth, {\em The method of Nehari manifold}, Handbook of nonconvex analysis and applications, 597--632, Int. Press, Somerville, MA, 2010.
	
	\bibitem{Ta} G. Tarantello, {\em On nonhomogeneous elliptic equations involving critical Sobolev exponent}. Ann. Inst. H Poincaré Anal. Non Linéaire 9(3), 281–304 (1992).
	
	\bibitem{T}	C. Tintarev, {\em Concentration Compactness: Functional-Analytic Theory of Concentration Phenomena.} De
	Gruyter, Berlin (2020). 
	
	
\end{thebibliography}
\end{document}